\documentclass[12pt]{amsart}

\usepackage[margin=1in]{geometry}
\usepackage{amsmath,amsthm,amssymb}
\usepackage{dsfont}
\usepackage{color}
\usepackage{hyperref}

\hypersetup{
    colorlinks=true, 
    linkcolor=blue, 
    urlcolor=red, 
    linktoc=all 
}

\input xy
\xyoption{all}

\newcommand{\C}{\mathbb{C}}
\newcommand{\PP}{\mathbb{P}}
\newcommand{\Z}{\mathbb{Z}}
\newcommand{\one}{\mathds{1}}
\newcommand{\G}{\mathcal{G}}
\newcommand{\map}{\textrm{Map}}
\newcommand{\sm}[1]{\langle{#1}\rangle}
\newcommand{\KK}{\tilde{K}}
\newcommand{\HH}[1]{\tilde{H}^{#1}}

\theoremstyle{plain}
\newtheorem{thm}{Theorem}[section]

\newtheorem{lemma}[thm]{Lemma}
\newtheorem{cor}[thm]{Corollary}

\begin{document}

\title{Homotopy types of $SU(n)$-gauge groups over non-spin 4-manifolds}
\author{Tseleung So}
\address{Mathematical Sciences, University of Southampton, SO17 1BJ, UK}
\email{tls1g14@soton.ac.uk}
\thanks{}



\date{}

\begin{abstract}
Let $M$ be an orientable, simply-connected, closed, non-spin 4-manifold and let~$\G_k(M)$ be the gauge group of the principal $G$-bundle over $M$ with second Chern class~\mbox{$k\in\Z$}. It is known that the homotopy type of $\G_k(M)$ is determined by the homotopy type of~$\G_k(\C\PP^2)$. In this paper we investigate properties of $\G_k(\C\PP^2)$ when $G=SU(n)$ that partly classify the homotopy types of the gauge groups.
\end{abstract}

\maketitle

\section{Introduction}
Let $G$ be a simple, simply-connected, compact Lie group and let $M$ be an orientable, simply-connected, closed 4-manifold. Then a principal $G$-bundle $P$ over $M$ is classified by its second Chern class $k\in\Z$. The associated gauge group $\G_k(M)$ is the topological group of~$G$-equivariant automorphisms of $P$ which fix~$M$.

When $M$ is a spin 4-manifold, topologists have been studying the homotopy types of gauge groups over $M$ extensively over the last twenty years. On the one hand, Theriault showed that \cite{theriault10a} there is a homotopy equivalence
\[
\G_k(M)\simeq\G_k(S^4)\times\prod^d_{i=1}\Omega^2G,
\]
where $d$ is the second Betti number of $M$. Therefore to study the homotopy type of $\G_k(M)$ it suffices to study $\G_k(S^4)$. On the other hand, many cases of homotopy types of $\G_k(S^4)$'s are known. For examples, there are 6 distinct homotopy types of $\G_k(S^4)$'s for $G=SU(2)$ \cite{kono91}, and 8 distinct homotopy types for $G=SU(3)$ \cite{HK06}. When localized rationally or at any prime, there are 16 distinct homotopy types for $G=SU(5)$ \cite{theriault15} and 8 distinct homotopy types for $G=Sp(2)$ \cite{theriault10b}.

When $M$ is a non-spin 4-manifold, the author showed that \cite{so16} there is a homotopy equivalence
\[
\G_k(M)\simeq\G_k(\C\PP^2)\times\prod^{d-1}_{i=1}\Omega^2G,
\]
so the homotopy type of $\G_k(M)$ depends on the special case $\G_k(\C\PP^2)$. Compared to the extensive work on $\G_k(S^4)$, only two cases of $\G_k(\C\PP^2)$ have been studied, which are the~$SU(2)$- and $SU(3)$-cases \cite{KT96, theriault12}. As a sequel to \cite{so16}, this paper investigates the homotopy types of~$\G_k(\C\PP^2)$'s in order to explore gauge groups over non-spin 4-manifolds.

A common approach to classifying the homotopy types of gauge groups is as follows. Atiyah, Bott and Gottlieb \cite{AB83, gottlieb72} showed that the classifying space $B\G_k(M)$ is homotopy equivalent to the connected component $\map_k(M, BG)$ of the mapping space $\map(M, BG)$ containing the map $k\alpha\circ q$, where $q:M\to S^4$ is the quotient map and $\alpha$ is a generator of~$\pi_4(BG)\cong\Z$. The evaluation map $ev:B\G_k(M)\to BG$ induces a fibration sequence
\begin{equation}\label{fib_Gk(M) ev}
\G_k(M)\longrightarrow G\overset{\partial_k}{\longrightarrow}\map^*_0(M, BG)\longrightarrow B\G_k(M)\overset{ev}{\longrightarrow}BG.
\end{equation}
For $M=S^4$, the order of $\partial_1:G\to\Omega^3_0G$ helps determine the classification of $\G_k(S^4)$'s by the following theorem. The first part is due to~\cite{theriault10b} and the second is due to~\cite{KKT14}.

\begin{thm}[Theriault, \cite{theriault10b}; Kishimoto, Kono, Tsutaya \cite{KKT14}]\label{thm_counting lemma S4}
Let $m$ be the order of $\partial_1$. Denote the $p$-component of $a$ by $\nu_p(a)$ and the greatest common divisor of $a$ and $b$ by $(a, b)$.
\begin{enumerate}
\item
If $(m,k)=(m,l)$, then $\G_k(S^4)$ is homotopy equivalent to $\G_l(S^4)$ when localized rationally or at any odd prime.
\item
If  $\G_k(S^4)$ is homotopy equivalent to $\G_l(S^4)$ and $G$ is of low rank (for details please see~\cite{KKT14}), then $\nu_p(m,k)=\nu_p(m,l)$ for any odd prime $p$.
\end{enumerate}
\end{thm}

Therefore the classification problem reduces to calculating the order $m$ of $\partial_1$. Known examples are $m=12$ for $G=SU(2)$ \cite{kono91}, $m=24$ for $G=SU(3)$ \cite{HK06}, $m=120$ for $G=SU(5)$ \cite{theriault15} and $m=40$ for $G=Sp(2)$ \cite{theriault10b}. For most cases of $G$, the exact value of $m$ is difficult to compute, but we are still able to obtain partial results. When $G$ is $SU(n)$, the order of $\partial_1$ and $n(n^2-1)$ have the same odd primary components if $n<(p-1)^2+1$ \cite{KKT14, theriault17}. Moreover, Hamanaka and Kono showed a necessary condition $(n(n^2-1),k)=(n(n^2-1),l)$ for a homotopy equivalence $\G_k(S^4)\simeq\G_l(S^4)$ \cite{HK06}.

In this paper we consider gauge groups over $\C\PP^2$. Take $M=\C\PP^2$ in~(\ref{fib_Gk(M) ev}) and denote the boundary map by $\partial'_k:G\to\map^*_0(\C\PP^2, BG)$. Since $\map^*_0(\C\PP^2, BG)$ is not an H-space,~$[G,\map^*_0(\C\PP^2,BG)]$ is not a group so the order of $\partial'_k$ makes no sense. However, we can still define an ``order'' of $\partial'_k$ \cite{theriault12}, which will be mentioned in Section~2. We show that the ``order'' of $\partial'_1$ helps determine the homotopy type of $\G_k(\C\PP^2)$ like part~(1) of Theorem~\ref{thm_counting lemma S4}.

\begin{thm}
Let $m'$ be the ``order'' of $\partial'_1$. If $(m', k)=(m', l)$, then $\G_k(\C\PP^2)$ is homotopy equivalent to $\G_l(\C\PP^2)$ when localized rationally or at any prime.
\end{thm}

In Section 4, we study the $SU(n)$-gauge groups over $\C\PP^2$ and use unstable $K$-theory to give a lower bound on the ``order'' of $\partial'_1$. 

\begin{thm}\label{thm_main thm}
When $G$ is $SU(n)$, the ``order'' of $\partial'_1$ is at least $\frac{1}{2}n(n^2-1)$ for $n$ odd, and~$n(n^2-1)$ for $n$ even.
\end{thm}

In Section 5, we prove a necessary condition for the homotopy equivalence~\mbox{$\G_k(\C\PP^2)\simeq\G_l(\C\PP^2)$} similar to that in~\cite{HK06}.

\begin{thm}\label{thm_necessary condition}
Let $G$ be $SU(n)$. If $\G_k(\C\PP^2)$ is homotopy equivalent to $\G_l(\C\PP^2)$, then
\[
\begin{cases}
(\frac{1}{2}n(n^2-1),k)=(\frac{1}{2}n(n^2-1),l),	&\text{for $n$ odd};\\
(n(n^2-1),k)=(n(n^2-1),l),							&\text{for $n$ even}.
\end{cases}
\]
\end{thm}

\section{Some facts about boundary map $\partial'_1$}
Take $M$ to be $S^4$ and $\C\PP^2$ respectively in fibration~(\ref{fib_Gk(M) ev}) to obtain fibration sequences
\begin{equation}\label{fib_Gk(S4)}
\G_k(S^4)\longrightarrow G\overset{\partial_k}{\longrightarrow}\Omega^3_0G\longrightarrow B\G_k(S^4)\overset{ev}{\longrightarrow}BG
\end{equation}
\begin{equation}\label{fib_Gk(CP2)}
\G_k(\C\PP^2)\longrightarrow G\overset{\partial'_k}{\longrightarrow}\map^*_0(\C\PP^2,BG)\longrightarrow B\G_k(\C\PP^2)\overset{ev}{\longrightarrow}BG.
\end{equation}
There is also a cofibration sequence
\begin{equation}\label{cofib_CP2}
S^3\overset{\eta}{\longrightarrow}S^2\longrightarrow\C\PP^2\overset{q}{\longrightarrow}S^4,
\end{equation}
where $\eta$ is Hopf map and $q$ is the quotient map. Due to the naturality of $q^*$, we combine fibrations~(\ref{fib_Gk(S4)}) and~(\ref{fib_Gk(CP2)}) to obtain a commutative diagram of fibration sequences
\begin{equation}\label{digm_dfn of tilde partial}
\xymatrix{
\G_k(S^4)\ar[d]^{q^*}\ar[r]	&G\ar@{=}[d]\ar[r]^{\partial_k}	&\Omega^3_0G\ar[d]^{q^*}\ar[r]		&B\G_k(S^4)\ar[d]^{q^*}\ar[r]	&BG\ar@{=}[d]\\
\G_k(\C\PP^2)\ar[r]			&G\ar[r]^-{\partial'_k}			&\map^*_0(\C\PP^2, BG)\ar[r]		&B\G_k(\C\PP^2)\ar[r]			&BG
}
\end{equation}
It is known that \cite{lang73} $\partial_k$ is triple adjoint to Samelson product
\[
\sm{k\imath,\one}:S^3\wedge G\overset{k\imath\wedge\one}{\longrightarrow}G\wedge G\overset{\sm{\one,\one}}{\longrightarrow}G,
\]
where $\imath:S^3\to SU(n)$ is the inclusion of the bottom cell and $\sm{\one,\one}$ is the Samelson product of the identity on $G$ with itself. The order of $\partial_k$ is its multiplicative order in the group~$[G, \Omega^3_0G]$.

Unlike $\Omega^3_0G$, $\map^*_0(\C\PP^2, BG)$ is not an H-space, so $\partial'_k$ has no order. In \cite{theriault12}, Theriault defined the ``order'' of $\partial'_k$ to be the smallest number $m'$ such that the composition
\[
G\overset{\partial_k}{\longrightarrow}\Omega^3_0G\overset{m'}{\longrightarrow}\Omega^3_0G\overset{q^*}{\longrightarrow}\map^*_0(\C\PP^2, BG)
\]
is null homotopic. In the following, we interpret the ``order'' of $\partial'_k$ as its multiplicative order in a group contained in $[\C\PP^2\wedge G, BG]$.

Apply $[-\wedge G, BG]$ to cofibration~(\ref{cofib_CP2}) to obtain an exact sequence of sets
\[
[\Sigma^3G, BG]\overset{(\Sigma\eta)^*}{\longrightarrow}[\Sigma^4G, BG]\overset{q^*}{\longrightarrow}[\C\PP^2\wedge G, BG].
\]
All terms except $[\C\PP^2\wedge G, BG]$ are groups and $(\Sigma\eta)^*$ is a group homomorphism since $\Sigma\eta$ is a suspension. We want to refine this exact sequence so that the last term is replaced by a group. Observe that $\C\PP^2$ is the cofiber of $\eta$ and so there is a coaction $\psi:\C\PP^2\to\C\PP^2\vee S^4$. We show that the coaction gives a group structure on $Im(q^*)$.

\begin{lemma}\label{lemma_dfn Im group}
Let $Y$ be a space and let $A\overset{f}{\to}B\overset{g}{\to}C\overset{h}{\to}\Sigma A$ be a cofibration sequence. If $\Sigma A$ is homotopy cocommutative, then $Im(h^*)$ is an abelian group and
\[
[\Sigma B, Y]\overset{(\Sigma f)^*}{\longrightarrow}[\Sigma A, Y]\overset{h^*}{\longrightarrow}Im(h^*)\longrightarrow0
\]
is an exact sequence of groups and group homomorphisms.
\end{lemma}

\begin{proof}
Apply $[-,Y]$ to the cofibration to get an exact sequence of sets
\begin{equation}\label{exact seq_set exact seq}
[\Sigma B, Y]\overset{(\Sigma f)^*}{\longrightarrow}[\Sigma A, Y]\overset{h^*}{\longrightarrow}[C, Y].
\end{equation}
Note that $[\Sigma B, Y]$ and $[\Sigma A, Y]$ are groups, and $(\Sigma f)^*$ is a group homomorphism. We will replace $[C, Y]$ by $Im(h^*)$ and define a group structure on it such that $h^*:[\Sigma A, Y]\to Im(h^*)$ is a group homomorphism.

For any $\alpha$ and $\beta$ in $[\Sigma A, Y]$, we define a binary operator $\boxtimes$ on $Im(h^*)$ by
\[
h^*\alpha\boxtimes h^*\beta=h^*(\alpha+\beta).
\]
To check this is well-defined we need to show $h^*(\alpha+\beta)\simeq h^*(\alpha'+\beta)\simeq h^*(\alpha+\beta')$ for any~$\alpha,\alpha',\beta,\beta'$ satisfying $h^*\alpha\simeq h^*\alpha'$ and $h^*\beta\simeq h^*\beta'$.

First we show $h^*(\alpha+\beta)\simeq h^*(\alpha'+\beta)$. By definition, we have
\[
h^*(\alpha+\beta)=(\alpha+\beta)\circ h=\triangledown\circ(\alpha\vee\beta)\circ\sigma\circ h,
\]
where $\sigma:\Sigma A\to\Sigma A\vee\Sigma A$ is the comultiplication and $\triangledown:Y\vee Y\to Y$ is the folding map. Since $C$ is a cofiber, there is a coaction $\psi:C\to C\vee\Sigma A$ such that $\sigma\circ h\simeq(h\vee\one)\circ\psi$.
\[\xymatrix{
C\ar[r]^-{\psi}\ar[d]^-{h}	&C\vee\Sigma A\ar[d]^-{h\vee\one}\\
\Sigma A\ar[r]^-{\sigma}	&\Sigma A\vee\Sigma A
}\]
Then we obtain a string of equivalences
\begin{eqnarray*}
h^*(\alpha+\beta)
&=&\triangledown\circ(\alpha\vee\beta)\circ\sigma\circ h\\
&\simeq&\triangledown\circ(\alpha\vee\beta)\circ(h\vee\one)\circ\psi\\
&\simeq&\triangledown\circ(\alpha'\vee\beta)\circ(h\vee\one)\circ\psi\\
&\simeq&\triangledown\circ(\alpha'\vee\beta)\circ\sigma\circ h\\
&=&h^*(\alpha'+\beta)
\end{eqnarray*}
The third line is due to the assumption $h^*\alpha\simeq h^*\alpha'$. Therefore we have $h^*(\alpha+\beta)\simeq h^*(\alpha'+\beta)$. Since $\Sigma A$ is cocommutative, $[\Sigma A, Y]$ is abelian and $h^*(\alpha+\beta)\simeq h^*(\beta+\alpha)$. Then we have
\[
h^*(\alpha+\beta)\simeq h^*(\beta+\alpha)\simeq h^*(\beta'+\alpha)\simeq h^*(\alpha+\beta').
\]
This implies $\boxtimes$ is well-defined.

Due to the associativity of $+$ in $[\Sigma A, Y]$, $\boxtimes$ is associative since
\begin{eqnarray*}
(h^*\alpha\boxtimes h^*\beta)\boxtimes h^*\gamma
&=&h^*(\alpha+\beta)\boxtimes h^*\gamma\\
&=&h^*((\alpha+\beta)+\gamma)\\
&=&h^*(\alpha+(\beta+\gamma))\\
&=&h^*\alpha\boxtimes h^*(\beta+\gamma)\\
&=&h^*\alpha\boxtimes (h^*\beta\boxtimes h^*\gamma).
\end{eqnarray*}
Clearly the trivial map $\ast:C\to Y$ is the identity of $\boxtimes$ and $h^*(-\alpha)$ is the inverse of $h^*\alpha$. Therefore $\boxtimes$ is indeed a group multiplication.

By definition of $\boxtimes$, $h^*:[\Sigma A, Y]\to Im(h^*)$ is a group homomorphism, and hence an epimorphism. Since $[\Sigma A, Y]$ is abelian, so is $Im(h^*)$. We replace $[C,Y]$ by $Im(h^*)$ in~(\ref{exact seq_set exact seq}) to obtain a sequence of groups and group homomorphisms
\[
[\Sigma B, Y]\overset{(\Sigma f)^*}{\longrightarrow}[\Sigma A, Y]\overset{h^*}{\longrightarrow}Im(h^*)\longrightarrow0.
\]
The exactness of~(\ref{exact seq_set exact seq}) implies $ker(h^*)=Im(\Sigma f)^*$, so the sequence is exact.
\end{proof}

Applying Lemma~\ref{lemma_dfn Im group} to cofibration $\Sigma^3G\to\Sigma^2G\to\C\PP^2\wedge G$ and the space $Y=BG$, we obtain an exact sequence of abelian groups
\begin{equation}\label{exact seq_general q^* refined}
[\Sigma^3G, BG]\overset{(\Sigma\eta)^*}{\longrightarrow}[\Sigma^4G, BG]\overset{q^*}{\longrightarrow}Im(q^*)\longrightarrow0.
\end{equation}
In the middle square of~(\ref{digm_dfn of tilde partial}) $\partial'_k\simeq q^*\partial_k$, so $\partial'_k$ is in $Im(q^*)$. For any number $m$,~\mbox{$q^*(m\partial_k)=mq^*\partial_k$}, so the ``order'' of $\partial'_k$ defined in \cite{theriault12} coincides with the multiplicative order of $\partial'_k$ in $Im(q^*)$. The exact sequence~(\ref{exact seq_general q^* refined}) allows us to compare the orders of $\partial_1$ and $\partial'_1$.

\begin{lemma}\label{lemma_partial'}
Let $m$ be the order of $\partial_1$ and let $m'$ be the order of $\partial'_1$. Then $m$ is $m'$ or $2m'$.
\end{lemma}

\begin{proof}
By exactness of~(\ref{exact seq_general q^* refined}), there is some $f\in[\Sigma^3G, BG]$ such that $(\Sigma\eta)^*f\simeq m'\partial_1$. Since~$\Sigma\eta$ has order 2, $2m'\partial_1$ is null homotopic. It follows that $2m'$ is a multiple of $m$. Since $m$ is greater than or equal to $m'$, $m$ is either $m'$ or $2m'$.
\end{proof}

When $G=SU(2)$, the order $m$ of $\partial_1$ is 12 and the order $m'$ of $\partial'_1$ is 6 \cite{KT96}. When~\mbox{$G=SU(3)$}, $m=24$ and $m'=12$ \cite{theriault12}. It is natural to ask whether $m=2m'$ for all~$G$. However, this is not the case. In a preprint by Theriault and the author, we showed that~$m=m'=40$ for~$G=Sp(2)$.

In the $S^4$ case, part~(1) of Theorem~\ref{thm_counting lemma S4} gives a sufficient condition for $\G_k(S^4)\simeq\G_l(S^4)$ when localized rationally or at any prime. In the $\C\PP^2$ case, Theriault showed a similar counting statement, in which the sufficient condition depends on the order of $\partial_1$ instead of~$\partial'_1$.

\begin{thm}[Theriault, \cite{theriault12}]\label{thm_stephen counting lemma}
Let $m$ be the order of $\partial_1$. If $(m,k)=(m,l)$, then $\G_k(\C\PP^2)$ is homotopy equivalent to $\G_l(\C\PP^2)$ when localized rationally or at any prime.
\end{thm}

Lemma~\ref{lemma_partial'} can be used to improve the sufficient condition of Theorem~\ref{thm_stephen counting lemma}.

\begin{thm}
Let $m'$ be the order of $\partial'_1$. If $(m',k)=(m',l)$, then $\G_k(\C\PP^2)$ is homotopy equivalent to $\G_l(\C\PP^2)$ when localized rationally or at any prime.
\end{thm}

\begin{proof}
By Lemma~\ref{lemma_partial'}, $m$ is either $m'$ or $2m'$. If $m=m'$, then the statement is same as Theorem~\ref{thm_stephen counting lemma}. If we localize rationally or at any odd prime, then $(m,k)=(m',k)$ for any~$k$, so a homotopy equivalence $\G_k(\C\PP^2)\simeq\G_l(\C\PP^2)$ follows by Theorem~\ref{thm_stephen counting lemma}. It remains to consider the case where $m=2m'$ when localized at 2.

Assume $m=2^n$ and $m'=2^{n-1}$. For any $k$, $(2^{n-1},k)=2^i$ where $i$ an integer such that~$0\leq i\leq n-1$. If $i\leq n-2$, then $k=2^it$ for some odd number $t$ and $(2^{n-1},k)=2^i$. The sufficient condition $(2^{n-1},k)=(2^{n-1},l)$ is equivalent to $(2^n,k)=(2^n,l)$. Again the homotopy equivalence $\G_k(\C\PP^2)\simeq\G_l(\C\PP^2)$ follows by Theorem~\ref{thm_stephen counting lemma}. If $i=n-1$, then $(2^n,k)$ is either~$2^n$ or $2^{n-1}$. We claim that $\G_k(\C\PP^2)$ has the same homotopy type for both $(2^n,k)=2^n$ or $(2^n,k)=2^{n-1}$.

Consider fibration~(\ref{fib_Gk(CP2)})
\[
\map^*_0(\C\PP^2, G)\longrightarrow\G_k(\C\PP^2)\longrightarrow G\overset{\partial'_k}{\longrightarrow}\map^*_0(\C\PP^2,BG).
\]
If $(2^n,k)=2^n$, then $k=2^nt$ for some number $t$. By linearity of Samelson products, $\partial_k\simeq k\partial_1$. Since $\partial'_k\simeq q^*k\partial_1\simeq q^*2^nt\partial_1$ and $\partial_1$ has order $2^n$, $\partial'_k$ is null homotopic and we have
\[
\G_k(\C\PP^2)\simeq G\times\map^*_0(\C\PP^2,G).
\]
If $(2^n,k)=2^{n-1}$, then $k=2^{n-1}t$ for some odd number $t$. Writing $t=2s+1$ gives~\mbox{$k=2^ns+2^{n-1}$}. Since $\partial'_k\simeq q^*k\partial_1\simeq q^*(2^ns+2^{n-1})\partial_1\simeq q^*2^{n-1}\partial_1$ and $\partial'_1$ has order $2^{n-1}$, $\partial'_k$ is null homotopic and we have
\[
\G_k(\C\PP^2)\simeq G\times\map^*_0(\C\PP^2,G).
\]
The same is true for $\G_l(\C\PP^2)$ and hence $\G_k(\C\PP^2)\simeq\G_l(\C\PP^2)$.
\end{proof}

\section{Plan for the proofs of Theorems~\ref{thm_main thm} and~\ref{thm_necessary condition}}

From this section onward, we will focus on $SU(n)$-gauge groups over $\C\PP^2$. There is a fibration
\begin{equation}\label{exact seq_U fibration}
SU(n)\longrightarrow SU(\infty)\overset{p}{\longrightarrow}W_n,
\end{equation}
where $p:SU(\infty)\to W_n$ is the projection and $W_n$ is the symmetric space $SU(\infty)/SU(n)$. Then we have
\begin{eqnarray*}
\HH{*}(SU(\infty))&=&\Lambda(x_3,\cdots,x_{2n-1},\cdots),\\
\HH{*}(SU(n))&=&\Lambda(x_3,\cdots,x_{2n-1}),\\
\HH{*}(BSU(n))&=&\Z[c_2,\cdots,c_n],\\
\HH{*}(W_n)&=&\Lambda(\bar{x}_{2n+1},\bar{x}_{2n+3},\cdots),
\end{eqnarray*}
where $x_{2n+1}$ has degree $2n+1$, $c_i$ is the $i^{\text{th}}$ universal Chern class and $x_{2i+1}=\sigma(c_{i+1})$ is the image of $c_{i+1}$ under the cohomology suspension $\sigma$, and $p^*(\bar{x}_{2i+1})=x_{2i+1}$. Furthermore,~$H^{2n}(\Omega W_n)\cong\Z$ and $H^{2n+2}(\Omega W_n)\cong\Z$ are generated by $a_{2n}$ and $a_{2n+2}$, where $a_{2i}$ is the transgression of $x_{2i+1}$.

The $(2n+4)$-skeleton of $W_n$ is $\Sigma^{2n-1}\C\PP^2$ for $n$ odd, and is $S^{2n+3}\vee S^{2n+1}$ for $n$ even, so its homotopy groups are as follows:
\begin{equation}\label{table_pi Wn}
\begin{array}{c|c c c c}
				&\multicolumn{4}{c}{\pi_i(W_n)}\\	\hline
i				&\leq 2n	&2n+1	&2n+2	&2n+3\\	\hline
n\text{ odd}	&0			&\Z		&0		&\Z\\
n\text{ even}	&0			&\Z		&\Z/2\Z	&\Z\oplus\Z/2\Z\\	
\end{array}
\end{equation}

The canonical map $\epsilon:\Sigma\C\PP^{n-1}\to SU(n)$ induces the inclusion~\mbox{$\epsilon_*:H_*(\Sigma\C\PP^{n-1})\to H_*(SU(n))$} of the generating set. Let $C$ be the quotient $\C\PP^{n-1}/\C\PP^{n-3}$ and let $\bar{q}:\Sigma\C\PP^{n-1}\to\Sigma C$ be the quotient map. Then there is a diagram
\[
\xymatrix{
[\Sigma C, SU(n)]\ar[r]^-{(\partial'_k)_*}\ar[d]^-{\bar{q}^*}	&[\Sigma C, \map^*(\C\PP^2,BSU(n))]\ar[r]\ar[d]^-{\bar{q}^*}	&[\Sigma C, B\G_k(\C\PP^2)]\ar[d]^-{\bar{q}^*}\\
[\Sigma\C\PP^{n-1}, SU(n)]\ar[r]^-{(\partial'_k)_*}				&[\Sigma\C\PP^{n-1}, \map^*(\C\PP^2,BSU(n))]\ar[r]				&[\Sigma\C\PP^{n-1}, B\G_k(\C\PP^2)],
}\]
where $(\partial'_k)_*$ sends $f$ to $\partial'_k\circ f$ and the rows are induced by fibration~(\ref{fib_Gk(CP2)}). In particular, in the second row the map $\epsilon:\Sigma\C\PP^{n-1}\to SU(n)$ is sent to $(\partial'_k)_*(\epsilon)=\partial'_k\circ\epsilon$. In Section 4, we use unstable $K$-theory to calculate the order of $\partial'_1\circ\epsilon$, giving a lower bound on the order of $\partial'_1$. Furthermore, in \cite{HK06} Hamanaka and Kono considered an exact sequence similar to the first row to give a necessary condition for $\G_k(S^4)\simeq\G_l(S^4)$. In Section 5 we follow the same approach and use the first row to give a necessary condition for $\G_k(\C\PP^2)\simeq\G_l(\C\PP^2)$.

We remark that it is difficult to use only one of the two rows to prove both Theorems~\ref{thm_main thm} and~\ref{thm_necessary condition}. On the one hand, $\partial'_1\circ\epsilon$ factors through a map $\bar{\partial}:\Sigma C\to\map^*(\C\PP^2, BSU(n))$. There is no obvious method to show that $\bar{\partial}$ and $\partial'_1\circ\epsilon$ have the same orders except direct calculation. Therefore we cannot compare the orders of $\bar{\partial}$ and $\partial'_1$ to prove Theorem~\ref{thm_main thm} without calculating the order of $\partial'_1\circ\epsilon$. On the other hand, applying the method used in Section 5 to the second row gives a much weaker conclusion than Theorem~\ref{thm_necessary condition}. This is because $[\Sigma C,B\G_k(\C\PP^2)]$ is a much smaller group than $[\Sigma\C\PP^{n-1},B\G_k(\C\PP^2)]$ and much information is lost by the map $\bar{q}^*$.

\section{A lower bound on the order of $\partial'_1$}

The restriction of $\partial_1$ to $\Sigma\C\PP^{n-1}$ is $\partial_1\circ\epsilon$, which is the triple adjoint of the composition
\[
\sm{\imath,\epsilon}:S^3\wedge\Sigma\C\PP^{n-1}\overset{\imath\wedge\epsilon}{\longrightarrow}SU(n)\wedge SU(n)\overset{\sm{\one,\one}}{\longrightarrow}SU(n).
\]
Since $SU(n)\simeq\Omega BSU(n)$, we can further take its adjoint and get
\[
\rho:\Sigma S^3\wedge\Sigma\C\PP^{n-1}\overset{\Sigma\imath\wedge\epsilon}{\longrightarrow}\Sigma SU(n)\wedge SU(n)\overset{[ev,ev]}{\longrightarrow}BSU(n),
\]
where $[ev,ev]$ is the Whitehead product of the evaluation map
\[
ev:\Sigma SU(n)\simeq\Sigma\Omega BSU(n)\to BSU(n)
\]
with itself. Similarly, the restriction $\partial'_1\circ\epsilon$ is adjoint to the composition
\[
\rho':\C\PP^2\wedge\Sigma\C\PP^{n-1}\overset{q\wedge\one}{\longrightarrow}S^4\wedge\Sigma\C\PP^{n-1}\overset{\Sigma\imath\wedge\epsilon}{\longrightarrow}\Sigma SU(n)\wedge SU(n)\overset{[ev,ev]}{\longrightarrow}BSU(n).
\]
Since we will frequently refer to the facts established in~\cite{HK03,HK06}, it is easier to follow their setting and consider its adjoint
\[
\gamma=\tau(\rho'\circ T):\C\PP^2\wedge\C\PP^{n-1}\to SU(n),
\]
where $T:\Sigma\C\PP^2\wedge\C\PP^{n-1}\to\C\PP^2\wedge\Sigma\C\PP^{n-1}$ is the swapping map and $\tau:[\Sigma\C\PP^2\wedge\C\PP^{n-1},BSU(n)]\to[\C\PP^2\wedge\C\PP^{n-1},SU(n)]$ is the adjunction. By adjunction, the orders of $\partial'_1\circ\epsilon,\rho'$ and $\gamma$ are the same. We will calculate the order of $\gamma$ using unstable $K$-theory to prove Theorem~\ref{thm_main thm}.

Apply $[\C\PP^2\wedge\C\PP^{n-1},-]$ to fibration~(\ref{exact seq_U fibration}) to obtain the exact sequence
\[
\KK^0(\C\PP^2\wedge\C\PP^{n-1})\overset{p_*}{\longrightarrow}[\C\PP^2\wedge\C\PP^{n-1},\Omega W_n]\longrightarrow[\C\PP^2\wedge\C\PP^{n-1},SU(n)]\longrightarrow\KK^1(\C\PP^2\wedge\C\PP^{n-1}).
\]
Since $\C\PP^2\wedge\C\PP^{n-1}$ is a CW-complex with even dimensional cells, $\KK^1(\C\PP^2\wedge\C\PP^{n-1})$ is zero. First we identify the term $[\C\PP^2\wedge\C\PP^{n-1},\Omega W_n]$.

\begin{lemma}\label{lemma_[Sigma CP^n-1,Omega W_n]}
We have the following:
\begin{itemize}
\item	$[\Sigma^{2n-4}\C\PP^2,\Omega W_n]\cong\Z$;
\item	$[\Sigma^{2n-3}\C\PP^2,\Omega W_n]=0$ for $n$ odd;
\item	$[\Sigma^{2n-2}\C\PP^2,\Omega W_n]\cong\Z\oplus\Z$.
\end{itemize}
\end{lemma}

\begin{proof}
First, apply $[\Sigma^{2n-4}-, \Omega W_n]$ to cofibration~(\ref{cofib_CP2}) to obtain the exact sequence
\[
\pi_{2n}(W_n)\longrightarrow\pi_{2n+1}(W_n)\longrightarrow[\Sigma^{2n-4}\C\PP^2,\Omega W_n]\longrightarrow\pi_{2n-1}(W_n).
\]
We refer to Table~(\ref{table_pi Wn}) freely for the homotopy groups of $W_n$. Since $\pi_{2n-1}(W_n)$ and $\pi_{2n}(W_n)$ are zero, $[\Sigma^{2n-4}\C\PP^{n-1},\Omega W_n]$ is isomorphic to $\pi_{2n+1}(W_n)\cong\Z$.

Second, apply $[\Sigma^{2n-3}-, \Omega W_n]$ to~(\ref{cofib_CP2}) to obtain
\[
\pi_{2n+2}(W_n)\longrightarrow[\Sigma^{2n-3}\C\PP^2,\Omega W_n]\longrightarrow\pi_{2n}(W_n).
\]
Since $\pi_{2n}(W_n)$ and $\pi_{2n+2}(W_n)$ are zero for $n$ odd, so is $[\Sigma^{2n-3}\C\PP^2,\Omega W_n]$.

Third, apply $[\Sigma^{2n-2}-, \Omega W_n]$ to~(\ref{cofib_CP2}) to obtain
\[
\pi_{2n+2}(W_n)\overset{\eta_1}{\longrightarrow}\pi_{2n+3}(W_n)\longrightarrow[\Sigma^{2n-2}\C\PP^2,\Omega W_n]\overset{j}{\longrightarrow}\pi_{2n+1}(W_n)\overset{\eta_2}{\longrightarrow}\pi_{2n+2}(W_n),
\]
where $\eta_1$ and $\eta_2$ are induced by Hopf maps $\Sigma^{2n}\eta:S^{2n+3}\to S^{2n+2}$ and \mbox{$\Sigma^{2n-1}\eta:S^{2n+2}\to S^{2n+1}$}, and $j$ is induced by the inclusion $S^{2n+1}\hookrightarrow\Sigma^{2n-2}\C\PP^2$ of the bottom cell. When $n$ is odd,~$\pi_{2n+2}(W_n)$ is zero and $\pi_{2n+1}(W_n)$ and $\pi_{2n+3}(W_n)$ are $\Z$, so $[\Sigma^{2n-2}\C\PP^{n-1},\Omega W_n]$ is $\Z\oplus\Z$. When $n$ is even, the $(2n+4)$-skeleton of $W_n$ is $S^{2n+1}\vee S^{2n+3}$. The inclusions
\[
\begin{array}{c c c}
i_1:S^{2n+1}\to S^{2n+1}\vee S^{2n+3}
&\text{and}
&i_2:S^{2n+3}\to S^{2n+1}\vee S^{2n+3}
\end{array}
\]
generate $\pi_{2n+1}(W_n)$ and the $\Z$-summand of $\pi_{2n+3}(W_n)$, and the compositions
\[
\begin{array}{c c c}
j_1:S^{2n+2}\overset{\Sigma^{2n-1}\eta}{\longrightarrow}S^{2n+1}\overset{i_1}{\longrightarrow}W_n
&\text{and}
&j_2:S^{2n+3}\overset{\Sigma^{2n}\eta}{\longrightarrow}S^{2n+2}\overset{\Sigma^{2n-1}\eta}{\longrightarrow}S^{2n+1}\overset{i_1}{\longrightarrow}W_n
\end{array}
\]
generate $\pi_{2n+2}(W_n)$ and the $\Z/2\Z$-summand of $\pi_{2n+3}(W_n)$ respectively. Since $\eta_1$ sends $j_1$ to~$j_2$, the cokernel of $\eta_1$ is $\Z$. Similarly, $\eta_2$ sends $i_1$ to $j_1$, so $\eta_2:\Z\to\Z/2\Z$ is surjective. This implies the preimage of $j$ is a $\Z$-summand. Therefore $[\Sigma^{2n-2}\C\PP^2,\Omega W_n]\cong\Z\oplus\Z$.
\end{proof}

Let $C$ be the quotient $\C\PP^{n-1}/\C\PP^{n-3}$. Since $\Omega W_n$ is $(2n-1)$-connected, $[\C\PP^2\wedge\C\PP^{n-1},\Omega W_n]$ is isomorphic to $[\C\PP^2\wedge C,\Omega W_n]$ which is easier to determine.

\begin{lemma}\label{lemma_[CP CP] is free}
The group $[\C\PP^2\wedge\C\PP^{n-1},\Omega W_n]\cong[\C\PP^2\wedge C,\Omega W_n]$ is isomorphic to $\Z^{\oplus3}$.
\end{lemma}

\begin{proof}
When $n$ is even, $C$ is $S^{2n-2}\vee S^{2n-4}$. By Lemma~\ref{lemma_[Sigma CP^n-1,Omega W_n]}, $[\C\PP^2\wedge C, \Omega W_n]$ is $[\Sigma^{2n-2}\C\PP^2, \Omega W_n]\oplus[\Sigma^{2n-4}\C\PP^2, \Omega W_n]\cong\Z^{\oplus3}$.

When $n$ is odd, $C$ is $\Sigma^{2n-6}\C\PP^2$. Apply $[\Sigma^{2n-6}\C\PP^2\wedge-, \Omega W_n]$ to cofibration~(\ref{cofib_CP2}) to obtain the exact sequence
\[
[\Sigma^{2n-3}\C\PP^2,\Omega W_n]\longrightarrow[\Sigma^{2n-2}\C\PP^2,\Omega W_n]\longrightarrow[\Sigma^{2n-6}\C\PP^2\wedge\C\PP^2,\Omega W_n]\longrightarrow
\]
\[
\longrightarrow[\Sigma^{2n-4}\C\PP^2,\Omega W_n]\longrightarrow[\Sigma^{2n-3}\C\PP^2,\Omega W_n]
\]
By Lemma~\ref{lemma_[Sigma CP^n-1,Omega W_n]}, the first and the last terms $[\Sigma^{2n-3}\C\PP^2, \Omega W_n]$ are zero, while the second term~$[\Sigma^{2n-2}\C\PP^2, \Omega W_n]$ is $\Z\oplus\Z$ and the fourth $[\Sigma^{2n-4}\C\PP^2, \Omega W_n]$ is $\Z$. Therefore~\mbox{$[\C\PP^2\wedge C,\Omega W_n]$} is $\Z^{\oplus3}$.
\end{proof}

Define $a:[\C\PP^2\wedge\C\PP^{n-1}, \Omega W_n]\to H^{2n}(\C\PP^2\wedge\C\PP^{n-1})\oplus H^{2n+2}(\C\PP^2\wedge\C\PP^{n-1})$ to be a map sending $f\in[\C\PP^2\wedge\C\PP^{n-1}, \Omega W_n]$ to $a(f)=f^*(a_{2n})\oplus f^*(a_{2n+2})$. The cohomology class $\bar{x}_{2n+1}$ represents a map $\bar{x}_{2n+1}:W_n\to K(\Z,2n+1)$ and $a_{2n}=\sigma(\bar{x}_{2n+1})$ represents its loop~\mbox{$\Omega\bar{x}_{2n+1}:\Omega W_n\to \Omega K(\Z,2n+1)$}. Similarly $a_{2n+2}=\sigma(\bar{x}_{2n+3})$ represents a loop map. This implies $a$ is a group homomorphism. Furthermore, $a_{2n}$ and $a_{2n+2}$ induce isomorphisms between $H^i(\Omega W_n)$ and $H^i(K(2n,\Z)\times K(2n+2,\Z))$ for $i=2n$ and $2n+2$. Since~\mbox{$[\C\PP^2\wedge\C\PP^{n-1},\Omega W_n]$} is a free $\Z$-module by Lemma~\ref{lemma_[CP CP] is free}, $a$ is a monomorphism. Consider the diagram
\begin{equation}\label{digm_injective Phi}
{\footnotesize\xymatrix{
\KK^0(\C\PP^2\wedge\C\PP^{n-1})\ar[r]^-{p_*}\ar@{=}[d]	&[\C\PP^2\wedge\C\PP^{n-1}, \Omega W_n]\ar[r]\ar[d]^-{a}						&[\C\PP^2\wedge\C\PP^{n-1}, SU(n)]\ar[d]^-{b}\ar[r]	&0\\
\KK^0(\C\PP^2\wedge\C\PP^{n-1})\ar[r]^-{\Phi}			&H^{2n}(\C\PP^2\wedge\C\PP^{n-1})\oplus H^{2n+2}(\C\PP^2\wedge\C\PP^{n-1})\ar[r]^-{\psi}	&Coker(\Phi)\ar[r]	&0
}}\end{equation}
In the left square, $\Phi$ is defined to be $a\circ p^*$. In the right square, $\psi$ is the quotient map and~$b$ is defined as follows. Any $f\in[\C\PP^2\wedge\C\PP^{n-1},SU(n)]$ has a preimage $\tilde{f}$ and $b(f)$ is defined to be $\psi(a(\tilde{f}))$. An easy diagram chase shows that $b$ is well-defined and injective. Since $b$ is injective, the order of $\gamma\in[\C\PP^2\wedge\C\PP^{n-1}, SU(n)]$ equals the order of $b(\gamma)\in Coker(\Phi)$. In \cite{HK03}, Hamanaka and Kono gave an explicit formula for $\Phi$.

\begin{thm}[Hamanaka, Kono, \cite{HK03}]\label{thm_Phi}
For any $f\in\KK^0(Y)$, we have
\[
\Phi(f)=n!ch_{2n}(f)\oplus(n+1)!ch_{2n+2}(f),
\]
where $ch_{2i}(f)$ is the $2i^{\text{th}}$ part of $ch(f)$.
\end{thm}

Let $u$ and $v$ be the generators of $H^2(\C\PP^2)$ and $H^2(\C\PP^{n-1})$. For $1\leq i\leq n-1$, denote $L_i$ and $L'_i$ as the generators of $\KK^0(\C\PP^2\wedge\C\PP^{n-1})$ with Chern characters $ch(L_i)=u^2(e^v-1)^i$ and $ch(L'_i)=(u+\frac{1}{2}u^2)\cdot(e^v-1)^i$. By Theorem~\ref{thm_Phi} we have
\begin{eqnarray*}
\Phi(L_i)
&=&n(n-1)A_iu^2v^{n-2}+n(n+1)B_iu^2v^{n-1},\\
\Phi(L'_i)&=&\frac{n(n-1)}{2}A_iu^2v^{n-2}+nB_iuv^{n-1}+\frac{n(n+1)}{2}B_iu^2v^{n-1},
\end{eqnarray*}
where
\[
\begin{array}{c c c}
A_i=\sum^i_{j=1}(-1)^{i+j}\binom{i}{j}j^{n-2}
&\text{and}
&B_i=\sum^i_{j=1}(-1)^{i+j}\binom{i}{j}j^{n-1}.
\end{array}
\]

Write an element $xu^2v^{n-2}+yuv^{n-1}+zu^2v^{n-1}\in H^{2n}(\C\PP^2\wedge\C\PP^{n-1})\oplus H^{2n+2}(\C\PP^2\wedge\C\PP^{n-1})$ as $(x, y, z)$. Then the coordinates of $\Phi(L_i)$ and $\Phi(L'_i)$ are $(n(n-1)A_i, 0, n(n+1)B_i)$ and~$(\frac{n(n-1)}{2}A_i, nB_i, \frac{n(n+1)}{2}B_i)$ respectively.

\begin{lemma}\label{lemma_simplify span of Im Phi}
For $n\geq3$, $Im(\Phi)$ is spanned by $(\frac{n(n-1)}{2},n,\frac{n(n+1)}{2})$, $(n(n-1),0,0)$ and $(0,2n,0)$.
\end{lemma}

\begin{proof}
By definition, $Im(\Phi)=span\{\Phi(L_i),\Phi(L'_i)\}^{n-1}_{i=1}$. For $i=1$, $A_1=B_1=1$. Then
\begin{eqnarray*}
\Phi(L_1)
&=&(n(n-1), 0, n(n+1))\\
&=&2(\frac{1}{2}n(n-1), n, \frac{1}{2}n(n+1))-(0,2n,0)\\
&=&2\Phi(L'_1)-(0,2n,0)
\end{eqnarray*}
Equivalently $(0,2n,0)=2\Phi(L'_1)-\Phi(L_1)$, so $span\{\Phi(L_1),\Phi(L'_1)\}=span\{\Phi(L'_1),(0,2n,0)\}$. For other $i$'s,
\begin{eqnarray*}
\Phi(L_i)
&=&(n(n-1)A_i, 0, n(n+1)B_i)\\
&=&2(\frac{1}{2}n(n-1)A_i, nB_i, \frac{1}{2}n(n+1)B_i)-(0,2nB_i,0)\\
&=&2\Phi(L'_i)-B_i(0,2n,0)
\end{eqnarray*}
is a linear combination of $\Phi(L'_i)$ and $(0, 2n, 0)$, so $Im(\Phi)=span\{\Phi(L'_1),\cdots,\Phi(L'_{n-1}),(0,2n,0)\}$.

We claim that $span\{\Phi(L'_i)\}^{n-1}_{i=1}=span\{\Phi(L'_1), (n(n-1),0,0)\}$. Observe that
\begin{eqnarray*}
\Phi(L'_i)
&=&(\frac{n(n-1)}{2}A_i,nB_i,\frac{n(n+1)}{2}B_i)\\
&=&(\frac{n(n-1)}{2}B_i,nB_i,\frac{n(n+1)}{2}B_i)+(\frac{n(n-1)}{2}(A_i-B_i),0,0)\\
&=&B_i\Phi(L'_1)+\frac{A_i-B_i}{2}\cdot(n(n-1),0,0).
\end{eqnarray*}
The difference
\begin{eqnarray*}
A_i-B_i
&=&\sum^i_{j=1}(-1)^{i+j}\binom{i}{j}j^{n-2}-\sum^i_{j=1}(-1)^{i+j}\binom{i}{j}j^{n-1}\\
&=&\sum^i_{j=1}(-1)^{i+j+1}\binom{i}{j}(j^{n-1}-j^{n-2})\\
&=&\sum^i_{j=1}(-1)^{i+j+1}\binom{i}{j}(j-1)j^{n-2}
\end{eqnarray*}
is even since each term $(j-1)j^{n-2}$ is even and $n\geq3$. Therefore $\frac{A_i-B_i}{2}$ is an integer and~$\Phi(L'_i)$ is a linear combination of $\Phi(L'_1)$ and $(n(n-1),0,0)$.

Furthermore,
\begin{eqnarray*}
\Phi(L'_2)
&=&B_2\Phi(L'_1)+(A_2-B_2)(\frac{n(n-1)}{2},0,0)\\
&=&B_2\Phi(L'_1)-2^{n-3}(n(n-1),0,0)
\end{eqnarray*}
and
\begin{eqnarray*}
\Phi(L'_3)
&=&B_3\Phi(L'_1)+(A_3-B_3)(\frac{n(n-1)}{2},0,0)\\
&=&B_3\Phi(L'_1)-(3^{n-2}-3\cdot2^{n-3})(n(n-1),0,0).
\end{eqnarray*}
Since $2^{n-3}$ and $3^{n-2}-3\cdot2^{n-3}$ are coprime to each other, there exist integers $s$ and $t$ such that $2^{n-3}s+(3^{n-2}-3\cdot2^{n-3})t=1$ and
\[
(n(n-1),0,0)=(sB_2+tB_3)\Phi(L'_1)-s\Phi(L'_2)-t\Phi(L'_3).
\]
Therefore $(n(n-1),0,0)$ is a linear combination of $\Phi(L'_1),\Phi(L'_2)$ and $\Phi(L'_3)$. This implies $span\{\Phi(L'_1),(n(n-1),0,0)\}=span\{\Phi(L'_i)\}^{n-1}_{i=1}$.

Combine all these together to obtain
\begin{eqnarray*}
Im(\Phi)
&=&span\{\Phi(L_i),\Phi(L'_i)\}^{n-1}_{i=1}\\
&=&span\{\Phi(L'_1),(n(n-1),0,0),(0,2n, 0)\}\\
&=&span\{(\frac{n(n-1)}{2},n,\frac{n(n+1)}{2}),(n(n-1),0,0),(0,2n,0)\}.
\end{eqnarray*}
\end{proof}

Back to diagram~(\ref{digm_injective Phi}). The map $\gamma$ has a lift $\tilde{\gamma}:\C\PP^2\wedge\C\PP^{n-1}\to\Omega W_n$. By exactness, the order of $\gamma$ equals the minimum number $m$ such that $m\tilde{\gamma}$ is contained in $Im(p_*)$. Since~$a$ and~$b$ are injective, the order of $\gamma$ equals the minimum number $m'$ such that $m'a(\tilde{\gamma})$ is contained in $Im(\Phi)$.

\begin{lemma}\label{lemma_lifting}
Let $\alpha:\Sigma X\to SU(n)$ be a map for some space $X$. If $\alpha':\C\PP^2\wedge X\to SU(n)$ is the adjoint of the composition
\[
\C\PP^2\wedge\Sigma X\overset{q\wedge\one}{\longrightarrow}\Sigma S^3\wedge\Sigma X\overset{\Sigma\imath\wedge\alpha}{\longrightarrow}\Sigma SU(n)\wedge SU(n)\overset{[ev,ev]}{\longrightarrow}BSU(n),
\]
then there is a lift $\tilde{\alpha}$ of $\alpha'$ such that $\tilde{\alpha}^*(a_{2i})=u^2\otimes\Sigma^{-1}\alpha^*(x_{2i-3})$, where $\Sigma$ is the cohomology suspension isomorphism.
\[
\xymatrix{
	&\Omega W_n\ar[d]\\
\C\PP^2\wedge X\ar[r]^-{\alpha'}\ar@{-->}[ur]^-{\tilde{\alpha}}	&SU(n)
}\]
\end{lemma}

\begin{proof}
In \cite{HK03,HK06}, Hamanaka and Kono constructed a lift $\Gamma:\Sigma SU(n)\wedge SU(n)\to W_n$ of $[ev,ev]$ such that $\Gamma^*(\bar{x}_{2i+1})=\sum_{j+k=i-1}\Sigma x_{2j+1}\otimes x_{2k+1}$. Let $\tilde{\Gamma}$ be the composition
\[
\tilde{\Gamma}:\C\PP^2\wedge\Sigma X\overset{q\wedge\one}{\longrightarrow}\Sigma S^3\wedge\Sigma X\overset{\Sigma\imath\wedge\alpha}{\longrightarrow}\Sigma SU(n)\wedge SU(n)\overset{\Gamma}{\longrightarrow}W_n.
\]
Then we have
\begin{eqnarray*}
\tilde{\Gamma}^*(\bar{x}_{2i+1})
&=&(q\wedge\one)^*(\Sigma\imath\wedge\alpha)^*\Gamma^*(\bar{x}_{2i+1})\\
&=&(q\wedge\one)^*(\Sigma\imath\wedge\alpha)^*\left(\sum_{j+k=i-1}\Sigma x_{2j+1}\otimes x_{2k+1}\right)\\
&=&(q\wedge\one)^*(\Sigma u_3\otimes\alpha^*(x_{2i-3}))\\
&=&u^2\otimes\alpha^*(x_{2i-3}),
\end{eqnarray*}
where $u_3$ is the generator of $H^3(S^3)$.

Let $T:\Sigma\C\PP^2\wedge X\to\C\PP^2\wedge\Sigma X$ be the swapping map and let $\tau:[\Sigma\C\PP^2\wedge X,W_n]\to[\C\PP^2\wedge X,\Omega W_n]$ be the adjunction. Take $\tilde{\alpha}:\C\PP^2\wedge X\to\Omega W_n$ to be the adjoint of $\tilde{\Gamma}$, that is $\tilde{\alpha}=\tau(\tilde{\Gamma}\circ T)$. Then $\tilde{\alpha}$ is a lift of $\alpha'$. Since
\[
(\tilde{\Gamma}\circ T)^*(\bar{x}_{2i+1})=T^*\circ\tilde{\Gamma}^*(\bar{x}_{2i+1})=T^*(u^2\otimes\alpha^*(x_{2i-3}))=\Sigma u^2\otimes\Sigma^{-1}\alpha^*(x_{2i-3}),
\]
we have $\tilde{\alpha}^*(a_{2i})=u^2\otimes\Sigma^{-1}\alpha^*(x_{2i-3})$.
\end{proof}

\begin{lemma}\label{lemma_tilde gamma coordinate}
In diagram~(\ref{digm_injective Phi}), $\gamma$ has a lift $\tilde{\gamma}$ such that $a(\tilde{\gamma})=u^2v^{n-2}\oplus u^2v^{n-1}$.
\end{lemma}

\begin{proof}
Recall that $\gamma$ is the adjoint of the composition
\[
\rho':\C\PP^2\wedge\Sigma\C\PP^{n-1}\overset{q\wedge\one}{\longrightarrow}\Sigma S^3\wedge\C\PP^{n-1}\overset{\Sigma\imath\wedge\epsilon}{\longrightarrow}\Sigma SU(n)\wedge SU(n)\overset{[ev,ev]}{\longrightarrow}BSU(n).
\]
Now we use Lemma~\ref{lemma_lifting} and take $\alpha$ to be $\epsilon:\Sigma\C\PP^{n-1}\to SU(n)$. Then $\gamma$ has a lift $\tilde{\gamma}$ such that $\tilde{\gamma}^*(a_{2i})=u^2\otimes\Sigma^{-1}\epsilon^*(x_{2i-3})=u^2\otimes v^{i-2}$. This implies
\[
a(\tilde{\gamma})=\tilde{\gamma}^*(a_{2n})\oplus\tilde{\gamma}^*(a_{2n+2})=u^2v^{n-2}\oplus u^2v^{n-1}.
\]
\end{proof}

Now we can calculate the order of $\partial'_1\circ\epsilon$, which gives a lower bound on the order of $\partial'_1$.

\begin{thm}\label{thm_order of gamma}
When $n\geq3$, the order of $\partial'_1\circ\epsilon$ is $\frac{1}{2}n(n^2-1)$ for $n$ odd and $n(n^2-1)$ for $n$ even.
\end{thm}

\begin{proof}
Since $\partial'_1\circ\epsilon$ is adjoint to $\gamma$ , it suffices to calculate the order of $\gamma$. By Lemma~\ref{lemma_simplify span of Im Phi},~$Im(\Phi)$ is spanned by $(\frac{1}{2}n(n-1),n,\frac{1}{2}n(n+1)),(n(n-1),0,0)$ and $(0,2n,0)$. By Lemma~\ref{lemma_tilde gamma coordinate}, $a(\tilde{\gamma})$ has coordinates $(1,0,1)$. Let $m$ be a number such that $ma(\tilde{\gamma})$ is contained in $Im(\Phi)$. Then
\[
m(1, 0, 1)=s(\frac{1}{2}n(n-1),n,\frac{1}{2}n(n+1))+t(n(n-1),0,0)+r(0,2n,0)
\]
for some integers $s, t$ and $r$. Solve this to get
\[
\begin{array}{c c c}
m=\frac{1}{2}tn(n^2-1),
&s=-2r,
&s=t(n-1).
\end{array}
\]
Since $s=-2r$ is even, the smallest positive value of $t$ satisfying $s=t(n-1)$ is 1 for $n$ odd and 2 for $n$ even. Therefore $m$ is $\frac{1}{2}n(n^2-1)$ for $n$ odd and $n(n^2-1)$ for $n$ even.
\end{proof}

For $SU(n)$-gauge groups over $S^4$, the order $m$ of $\partial_1$ has the form $m=n(n^2-1)$ for $n=3$ and $5$ \cite{HK06,theriault15}. If $p$ is an odd prime and $n<(p-1)^2+1$, then $m$ and $n(n^2-1)$ have the same $p$-components \cite{KKT14,theriault17}. These facts suggest it may be the case that $m=n(n^2-1)$ for any~$n>2$. In fact, one can follow the method Hamanaka and Kono used in \cite{HK06} and calculate the order of $\partial\circ\epsilon$ to obtain a lower bound $n(n^2-1)$ for $n$ odd. However, it does not work for the $n$ even case since $[S^4\wedge\C\PP^{n-1},\Omega W_n]$ is not a free $\Z$-module. An interesting corollary of Theorem~\ref{thm_order of gamma} is to give a lower bound on the order of $\partial_1$ for $n$ even.

\begin{cor}
When $n$ is even and greater than 2, the order of $\partial_1$ is at least $n(n^2-1)$.
\end{cor}

\begin{proof}
The order of $\partial'_1\circ\epsilon$ is a lower bound on the order of $\partial'_1$, which is either the same as or half of the order of $\partial_1$ by Lemma~\ref{lemma_partial'}. The corollary follows from Theorem~\ref{thm_order of gamma}.
\end{proof}

\section{A necessary condition for $\G_k(\C\PP^2)\simeq\G_l(\C\PP^2)$}
In this section we follow the approach in \cite{HK06} to prove Theorem~\ref{thm_necessary condition}. The techniques used are similar to that in Section 4, except we are working with the quotient $\Sigma C=\Sigma\C\PP^{n-1}/\Sigma\C\PP^{n-1}$ instead of $\Sigma\C\PP^{n-1}$. When $n$ is odd, $C$ is $\Sigma^{2n-6}\C\PP^2$, and when $n$ is even, $C$ is $S^{2n-2}\vee S^{2n-4}$. Apply $[\Sigma C,-]$ to fibration~(\ref{fib_Gk(CP2)}) to obtain the exact sequence
\[
[\Sigma C, SU(n)]\overset{(\partial'_k)_*}{\longrightarrow}[\Sigma C, \map^*_0(\C\PP^2,BSU(n))]\longrightarrow[\Sigma C, B\G_k(\C\PP^2)]\longrightarrow[\Sigma C, BSU(n)],
\]
where $(\partial'_k)_*$ sends $f\in[\Sigma C, SU(n)]$ to $\partial'_k\circ f\in[\Sigma C, \map^*_0(\C\PP^2, BSU(n))]$. Since $BSU(n)\to BSU(\infty)$ is a $2n$-equivalence and $\Sigma C$ has dimension $2n-1$, $[\Sigma C, BSU(n)]$ is $\KK^0(\Sigma C)$ which is zero. Similarly, $[\Sigma C, SU(n)]\cong[\Sigma^2C, BSU(n)]$ is $\KK^0(\Sigma^2C)\cong\Z\oplus\Z$. Furthermore, by adjunction we have $[\Sigma C, \map^*_0(\C\PP^2,BSU(n))]\cong[\Sigma C\wedge\C\PP^2, BSU(n)]$. The exact sequence becomes
\begin{equation}\label{exact seq_necessary Sigma C}
\KK^0(\Sigma^2C)\overset{(\partial'_k)_*}{\longrightarrow}[\Sigma C\wedge\C\PP^2, BSU(n)]\longrightarrow[\Sigma C, B\G_k(\C\PP^2)]\longrightarrow0.
\end{equation}
This implies $[\Sigma C, B\G_k(\C\PP^2)]\cong[C, \G_k(\C\PP^2)]$ is $Coker(\partial'_k)_*$. Also, apply $[\C\PP^2\wedge C, -]$ to fibration~(\ref{exact seq_U fibration}) to obtain the exact sequence
\begin{equation}\label{exact seq_necessary CP2 C}
[\C\PP^2\wedge C, \Omega SU(\infty)]\overset{p_*}{\longrightarrow}[\C\PP^2\wedge C, \Omega W_n]\longrightarrow[\C\PP^2\wedge C, SU(n)]\longrightarrow[\C\PP^2\wedge C, SU(\infty)].
\end{equation}

Observe that $[\C\PP^2\wedge C, \Omega SU(\infty)]\cong\KK^0(\C\PP^2\wedge C)$ is $\Z^{\oplus4}$ and~\mbox{$[\C\PP^2\wedge C, SU(\infty)]\cong\KK^1(\C\PP^2\wedge C)$} is zero. Combine exact sequences~(\ref{exact seq_necessary Sigma C}) and~(\ref{exact seq_necessary CP2 C}) to obtain the diagram
\[
\xymatrix{
	&\KK^0(\C\PP^2\wedge C)\ar[d]^-{p_*}\ar[dr]^-{\Phi}	&	&\\
	&[\C\PP^2\wedge C, \Omega W_n]\ar[d]\ar[r]^-{a}	&H^{2n}(\C\PP^2\wedge C)\oplus H^{2n+2}(\C\PP^2\wedge C)	&\\
\KK^0(\Sigma^2C)\ar[r]^-{(\partial'_k)_*}	&[\C\PP^2\wedge C, SU(n)]\ar[r]\ar[d]	&[C, B\G_k(\C\PP^2)]\ar[r]	&0\\
	&0	&	&
}\]
where $a(f)=f^*(a_{2n})\oplus f^*(a_{2n+2})$ for any $f\in[\C\PP^2\wedge C, \Omega W_n]$, and $\Phi$ is defined to be $a\circ p_*$. By Lemma~\ref{lemma_[CP CP] is free} $[\C\PP^2\wedge C,\Omega W_n]$ is free. Following the same argument in Section 4 implies the injectivity of $a$.

Our strategy to prove Theorem~\ref{thm_necessary condition} is as follows. If $\G_k(\C\PP^2)$ is homotopy equivalent to~$\G_l(\C\PP^2)$, then $[C, \G_k(\C\PP^2)]\cong[C, \G_l(\C\PP^2)]$ and exactness in~(\ref{exact seq_necessary CP2 C}) implies that $Im(\partial'_k)_*$ and~$Im(\partial'_l)_*$ have the same order in $[\C\PP^2\wedge C, SU(n)]$, resulting in a necessary condition for a homotopy equivalence $\G_k(\C\PP^2)\simeq\G_l(\C\PP^2)$. To calculate the order of $Im(\partial'_k)_*$, we will find a preimage $\tilde{\partial}_k$ of $Im(\partial'_k)_*$ in $[\C\PP^2\wedge C,\Omega W_n]$. Since $a$ is injective, we can embed $\tilde{\partial}_k$ into~$H^{2n}(\C\PP^2\wedge C)\oplus H^{2n+2}(\C\PP^2\wedge C)$ and work out the order of $Im(\partial'_k)_*$ there.

Let $u, v_{2n-4}$ and $v_{2n-2}$ be generators of $H^2(\C\PP^2)$, $H^{2n-4}(C)$ and $H^{2n-2}(C)$. Then we write an element $xu^2v_{2n-4}+yuv_{2n-2}+zu^2v_{2n-2}\in H^{2n}(\C\PP^2\wedge C)\oplus H^{2n+2}(\C\PP^2\wedge C)$ as $(x, y, z)$. First we need to find the submodule $Im(a)$.

\begin{lemma}\label{lemma_im lambda}
For $n$ odd, $Im(a)$ is $\{(x, y, z)|x+y\equiv z\pmod{2}\}$, and for $n$ even, $Im(a)$ is~$\{(x, y, z)|y\equiv0\pmod{2}\}$.
\end{lemma}

\begin{proof}
When $n$ is odd, $C$ is $\Sigma^{2n-6}\C\PP^2$ and the $(2n+3)$-skeleton of $\Omega W_n$ is $\Sigma^{2n-2}\C\PP^2$. To say $(x, y, z)\in Im(a)$ means there exists $f\in[\C\PP^2\wedge C,\Omega W_n]$ such that
\begin{equation}\label{equation in proof of lemma Im a}
\begin{array}{c c c}
f^*(a_{2n})=xu^2v_{2n-4}+yuv_{2n-2}
&\text{and}
&f^*(a_{2n+2})=zu^2v_{2n-2}.
\end{array}
\end{equation}
Reducing to homology with $\Z/2\Z$-coefficients, we have
\[
\begin{array}{c c c}
Sq^2(u)=u^2,
&Sq^2(v_{2n-4})=v_{2n-2},
&Sq^2(a_{2n})=a_{2n+2}.
\end{array}
\]
Apply $Sq^2$ to~(\ref{equation in proof of lemma Im a}) to get $x+y\equiv z\pmod{2}$. Therefore $Im(a)$ is contained in $\{(x, y, z)|x+y\equiv z\pmod{2}\}$. To show that they are equal, we need to show that $(1, 0, 1),(0, 1, 1)$ and $(0, 0, 2)$ are in $Im(a)$. Consider maps
\[
\begin{array}{l}
f_1:\C\PP^2\wedge C\overset{q_1}{\longrightarrow}S^4\wedge C\simeq\Sigma^{2n-2}\C\PP^2\hookrightarrow\Omega W_n\\
f_2:\C\PP^2\wedge C\overset{q_2}{\longrightarrow}\C\PP^2\wedge S^{2n-2}\hookrightarrow\Omega W_n\\
f_3:\C\PP^2\wedge C\overset{q_3}{\longrightarrow}S^{2n+2}\overset{\theta}{\longrightarrow}\Omega W_n
\end{array}
\]
where $q_1,q_2$ and $q_3$ are quotient maps and $\theta$ is the generator of $\pi_{2n+3}(W_n)$. Their images are
\[
\begin{array}{c c c}
a(f_1)=(1, 0, 1)
&a(f_2)=(0, 1, 1)
&a(f_3)=(0, 0, 2)
\end{array}
\]
respectively, so $Im(a)=\{(x, y, z)|x+y\equiv z\pmod{2}\}$.

When $n$ is even, $C$ is $S^{2n-2}\vee S^{2n-4}$ and the $(2n+3)$-skeleton of $\Omega W_n$ is $S^{2n+2}\vee S^{2n}$. Reducing to homology with $\Z/2\Z$-coefficients, $Sq^2(v_{2n-4})=0$ and $Sq^2(a_{2n})=0$. Apply $Sq^2$ to~(\ref{equation in proof of lemma Im a}) to get $y\equiv 0\pmod{2}$. Therefore $Im(a)$ is contained in $\{(x, y, z)|y\equiv0\pmod{2}\}$. To show that they are equal, we need to show that $(1, 0, 0),(0, 2, 0)$ and $(0, 0, 1)$ are in $Im(a)$. The maps
\[
\begin{array}{l}
f'_1:\C\PP^2\wedge C\overset{q'_1}{\longrightarrow}S^4\wedge(S^{2n-2}\vee S^{2n-4})\overset{p_1}{\longrightarrow}S^4\wedge S^{2n-4}\hookrightarrow\Omega W_n\\
f'_2:\C\PP^2\wedge C\overset{q'_2}{\longrightarrow}S^4\wedge(S^{2n-2}\vee S^{2n-4})\overset{p_2}{\longrightarrow}S^4\wedge S^{2n-2}\hookrightarrow\Omega W_n
\end{array}
\]
where $q'_1$ and $q'_2$ are quotient maps and $p_1$ and $p_2$ are pinch maps, have images $a(f'_1)=(1, 0, 0)$ and $a(f'_2)=(0, 0, 1)$. To find $(0, 2, 0)$, apply $[-\wedge S^{2n-2},\Omega W_n]$ to cofibration~(\ref{cofib_CP2}) to obtain the exact sequence
\[
\pi_{2n+3}(W_n)\longrightarrow[\C\PP^2\wedge S^{2n-2},\Omega W_n]\overset{i^*}{\longrightarrow}\pi_{2n+1}(W_n)\overset{\eta^*}{\longrightarrow}\pi_{2n+2}(W_n)
\]
where $i^*$ is induced by the inclusion $i:S^2\hookrightarrow\C\PP^2$ and $\eta^*$ is induced by Hopf map $\eta$. The third term $\pi_{2n+1}(W_n)\cong\Z$ is generated by $i':S^{2n+1}\to W_n$, the inclusion of the bottom cell, and the fourth term $\pi_{2n+2}(W_n)\cong\Z/2\Z$ is generated by $i'\circ\eta$, so $\eta^*:\Z\to\Z/2\Z$ is a surjection. By exactness $[\C\PP^2\wedge S^{2n-2},\Omega W_n]$ has a $\Z$-summand with the property that $i^*$ sends its generator $g$ to $2i'$. Therefore the composition
\[
f'_3:\C\PP^2\wedge(S^{2n-2}\vee S^{2n-4})\overset{pinch}{\longrightarrow}\C\PP^2\wedge S^{2n-2}\overset{g}{\longrightarrow}\Omega W_n
\]
has image $(0,2,0)$. It follows that $Im(a)=\{(x, y, z)|y\equiv0\pmod{2}\}$.
\end{proof}

Now we split into the $n$ odd and $n$ even cases to calculate the order of $Im(\partial'_k)_*$.

\subsection{The order of $Im(\partial'_k)_*$ for $n$ odd}
When $n$ is odd, $C$ is $\Sigma^{2n-6}\C\PP^2$. First we find~$Im(\Phi)$ in $Im(a)$. For $1\leq i\leq4$, let $L_i$ be the generators of $\KK^0(\C\PP^2\wedge C)\cong\Z^{\oplus 4}$ with Chern characters
\[
\begin{array}{l l}
ch(L_1)=(u+\frac{1}{2}u^2)\cdot(v_{2n-4}+\frac{1}{2}v_{2n-2})
&ch(L_2)=(u+\frac{1}{2}u^2)v_{2n-2}\\[15pt]
ch(L_3)=u^2(v_{2n-4}+\frac{1}{2}v_{2n-2})
&ch(L_4)=u^2v_{2n-2}.
\end{array}
\]
By Theorem~\ref{thm_Phi}, we have
\begin{eqnarray*}
\Phi(L_1)&=&\frac{n!}{2}u^2v_{2n-4}+\frac{n!}{2}uv_{2n-2}+\frac{(n+1)!}{4}u^2v_{2n-2}\\
\Phi(L_2)&=&n!uv_{2n-2}+\frac{(n+1)!}{2}u^2v_{2n-2}\\
\Phi(L_3)&=&n!u^2v_{2n-4}+\frac{(n+1)!}{2}u^2v_{2n-2}\\
\Phi(L_4)&=&(n+1)!u^2v_{2n-2}.
\end{eqnarray*}

By Lemma~\ref{lemma_im lambda}, $Im(a)$ is spanned by $(1, 0, 1), (0, 1, 1)$ and $(0, 0, 2)$. Under this basis, the coordinates of the $\Phi(L_i)$'s are
\[
\begin{array}{l l}
\Phi(L_1)=(\frac{n!}{2},\frac{n!}{2},\frac{(n-3)\cdot n!}{8}),
&\Phi(L_2)=(0,n!,\frac{(n-1)\cdot n!}{4}),\\[15pt]
\Phi(L_3)=(n!,0,\frac{(n-1)\cdot n!}{4}),
&\Phi(L_4)=(0,0,\frac{(n+1)!}{2}).
\end{array}
\]
We represent their coordinates by the matrix
\[
M_{\Phi}=L
\begin{pmatrix}
\frac{n(n-1)}{2}	&\frac{n(n-1)}{2}	&\frac{n(n-1)(n-3)}{8}\\
0					&n(n-1)				&\frac{n(n-1)^2}{4}\\
n(n-1)				&0					&\frac{n(n-1)^2}{4}\\
0					&0					&\frac{n(n^2-1)}{2}
\end{pmatrix},
\]
where $L=(n-2)!$. Then $Im(\Phi)$ is spanned by the row vectors of $M_{\Phi}$.

Next, we find a preimage of $Im(\partial'_k)_*$ in $[\C\PP^2\wedge C,\Omega W_n]$. In exact sequence~(\ref{exact seq_necessary Sigma C}) $\KK^0(\Sigma^2C)$ is $\Z\oplus\Z$. Let $\alpha_1$ and $\alpha_2$ be its generators with Chern classes
\[
\begin{array}{l l}
c_{n-1}(\alpha_1)=(n-2)!\Sigma^2v_{2n-4}	&c_n(\alpha_1)=\frac{(n-1)!}{2}\Sigma^2v_{2n-2}\\
c_{n-1}(\alpha_2)=0							&c_n(\alpha_2)=(n-1)!\Sigma^2v_{2n-2}.
\end{array}
\]

\begin{lemma}\label{lemma_lift of alpha}
For $i=1,2$, $\xi_k(\alpha_i)$ has a lift $\tilde{\alpha}_{i,k}:\C\PP^2\wedge C\to\Omega W_n$ such that
\[
a(\tilde{\alpha}_{i,k})=ku^2\otimes\Sigma^{-2}c_{n-1}(\alpha_i)\oplus ku^2\otimes\Sigma^{-2}c_{n}(\alpha_i).
\]
\end{lemma}

\begin{proof}
For dimension and connectivity reasons, $\alpha_i:\Sigma^2C\to BSU(\infty)$ lifts through $BSU(n)\to BSU(\infty)$. Label the lift $\Sigma^2C\to BSU(n)$ by $\alpha_i$ as well. Let $\alpha'_i:\Sigma C\to SU(n)$ be the adjoint of $\alpha_i$. Then $(\partial'_k)_*(\alpha_i)$ is the adjoint of the composition
\[
\C\PP^2\wedge\Sigma C\overset{q\wedge\one}{\longrightarrow}\Sigma S^3\wedge\Sigma C\overset{\Sigma k\imath\wedge\alpha'_i}{\longrightarrow}\Sigma SU(n)\wedge SU(n)\overset{[ev,ev]}{\longrightarrow}BSU(n).
\]
By Lemma~\ref{lemma_lifting}, $(\partial'_k)_*(\alpha_i)$ has a lift $\tilde{\alpha}_{i,k}$ such that $\tilde{\alpha}_{i,k}^*(a_{2j})=ku^2\otimes\Sigma^{-1}(\alpha')^*(x_{2j-3})$. Since $\sigma(c_{j-1})=x_{2j-3}$, we have $\tilde{\alpha}_{i,k}^*(a_{2j})=ku^2\otimes\Sigma^{-2}c_{j-1}(\alpha_i)$ and
\[
a(\tilde{\alpha}_{i,k})=ku^2\otimes\Sigma^{-2}c_{n-1}(\alpha_i)\oplus ku^2\otimes\Sigma^{-2}c_{n}(\alpha_i).
\]
\end{proof}

By Lemma~\ref{lemma_lift of alpha}, $(\partial'_k)_*(\alpha_1)$ and $(\partial'_k)_*(\alpha_2)$ have lifts
\[
\begin{array}{c c c}
\tilde{\alpha}_{1,k}=(n-2)!ku^2v_{2n-4}+\frac{(n-1)!}{2}ku^2v_{2n-2}
&\text{and}
&\tilde{\alpha}_{2,k}=(n-1)!ku^2v_{2n-2}.
\end{array}
\]
We represent their coordinates by the matrix
\[
M_{\partial}=kL
\begin{pmatrix}
1	&0	&\frac{n-3}{4}\\
0	&0	&\frac{n-1}{2}
\end{pmatrix}.
\]
Let $\tilde{\partial}_k=span\{\tilde{\alpha}_{1,k}, \tilde{\alpha}_{2,k}\}$ be the preimage of $Im(\partial'_k)_*$ in $[\C\PP^2\wedge C,\Omega W_n]$. Then $\tilde{\partial}_k$ is spanned by the row vectors of $M_{\partial}$.

\begin{lemma}\label{lemma_im xi 4m+3}
When $n$ is odd, the order of $Im(\partial'_k)_*$ is
\[
|Im(\partial'_k)_*|=\frac{\frac{1}{2}n(n^2-1)}{(\frac{1}{2}n(n^2-1), k)}\cdot\frac{n}{(n,k)}.
\]
\end{lemma}

\begin{proof}
Suppose $n=4m+3$ for some integer $m$. Then
\[
M_{\Phi}=(4m+3)L
\begin{pmatrix}
2m+1	&2m+1	&2m^2+m\\
0		&4m+2	&4m^2+4m+1\\
4m+2	&0		&4m^2+4m+1\\
0		&0		&8m^2+12m+4
\end{pmatrix}
\]
and
\[
M_{\partial}=kL
\begin{pmatrix}
1	&0	&m\\
0	&0	&2m+1
\end{pmatrix}.
\]
Transform $M_{\Phi}$ into Smith normal form
\[
A\cdot M_{\Phi}\cdot B=(4m+3)L
\begin{pmatrix}
(2m+1)	&	&\\
	&(2m+1)	&\\
	&	&(2m+1)(4m+4)\\
	&	&0
\end{pmatrix},
\]
where
\[
\begin{array}{c c c}
A=
\begin{pmatrix}
1		&0		&0			&0\\
-2		&0		&1			&0\\
4m+2	&1		&-(2m+1)	&0\\
4		&-2		&-2			&1
\end{pmatrix}
&\text{and}
&B=
\begin{pmatrix}
1	&-m	&-(2m+1)\\
0	&0	&1\\
0	&1	&2
\end{pmatrix}.
\end{array}
\]
The matrix $B$ represents a basis change in $Im(a)$ and $A$ represents a basis change in $Im(\Phi)$. Therefore $[\C\PP^2\wedge C, SU(n)]$ is isomorphic to
\[
\frac{\Z}{\frac{1}{2}(4m+3)!\Z}\oplus\frac{\Z}{\frac{1}{2}(4m+3)!\Z}\oplus\frac{\Z}{\frac{1}{2}(4m+4)!\Z}.
\]
We need to find the representation of $\tilde{\partial}_k$ under the new basis represented by $B$. The new coordinates of $\tilde{\alpha}_{1,k}$ and $\tilde{\alpha}_{2,k}$ are the row vectors of the matrix
\[
M_{\partial}
\cdot
\begin{pmatrix}
1	&-m	&-(2m+1)\\
0	&0	&1\\
0	&1	&2
\end{pmatrix}
=
\begin{pmatrix}
kL	&0			&-kL\\
0	&(2m+1)kL	&(4m+2)kL
\end{pmatrix}.
\]
Apply row operations to get
\[
\begin{pmatrix}
1		&0\\
4m+2	&1
\end{pmatrix}
\cdot
\begin{pmatrix}
kL	&0			&-kL\\
0	&(2m+1)kL	&(4m+2)kL
\end{pmatrix}
=
\begin{pmatrix}
kL			&0			&-kL\\
(4m+2)kL	&(2m+1)kL	&0
\end{pmatrix}.
\]
Let $\mu=(kL,0,-kL)$ and $\nu=((4m+2)kL, (2m+1)kL,0)$. Then
\[
\tilde{\partial}_k=\{x\mu+y\nu\in[\C\PP^2\wedge C,\Omega W_n]|x, y\in\Z\}.
\]
If $x\mu+y\nu$ and $x'\mu+y'\nu$ are the same in $Im(\Phi)$, then we have
\[\left\{
\begin{array}{r c l l}
xkL+(4m+2)ykL	&\equiv &x'kL+(4m+2)y'kL	&\pmod{(2m+1)(4m+3)L}\\
(2m+1)ykL		&\equiv	&(2m+1)y'kL			&\pmod{(2m+1)(4m+3)L}\\
xkL				&\equiv	&x'kL				&\pmod{(2m+1)(4m+3)(4m+4)L}
\end{array}\right.
\]
These conditions are equivalent to
\[\left\{
\begin{array}{r c l l}
xk	&\equiv &x'k	&\pmod{(2m+2)(4m+3)(4m+2)}\\
yk	&\equiv	&y'k	&\pmod{(4m+3)}
\end{array}\right.
\]
This implies that there are $\displaystyle{\frac{(2m+2)(4m+3)(4m+2)}{((2m+2)(4m+3)(4m+2), k)}}$ distinct values of $x$ and $\displaystyle{\frac{4m+3}{(4m+3,k)}}$ distinct values of $y$, so we have
\[
|Im(\partial'_k)_*|=\frac{(2m+2)(4m+3)(4m+2)}{((2m+2)(4m+3)(4m+2), k)}\cdot\frac{4m+3}{(4m+3,k)}.
\]

When $n=4m+1$, we can repeat the calculation above to obtain
\[
|Im(\partial'_k)_*|=\frac{2m(4m+2)(4m+1)}{(2m(4m+2)(4m+1), k)}\cdot\frac{4m+1}{(4m+1,k)}.
\]
\end{proof}

\subsection{The order of $Im(\partial'_k)_*$ for $n$ even}

When $n$ is even, $C$ is $S^{2n-2}\vee S^{2n-4}$. For $1\leq i\leq4$, let $L_i$ be the generators of $\KK^0(\C\PP^2\wedge C)\cong\Z^{\oplus 4}$ with Chern characters
\[
\begin{array}{l l}
ch(L_1)=(u+\frac{1}{2}u^2)v_{2n-4}
&ch(L_2)=u^2v_{2n-4}\\[15pt]
ch(L_3)=(u+\frac{1}{2}u^2)v_{2n-2}
&ch(L_4)=u^2v_{2n-2}.
\end{array}
\]
By Theorem~\ref{thm_Phi}, we have
\begin{eqnarray*}
\Phi(L_1)&=&\frac{n!}{2}u^2v_{2n-4}\\
\Phi(L_2)&=&n!u^2v_{2n-4}\\
\Phi(L_3)&=&n!uv_{2n-2}+\frac{(n+1)!}{2}u^2v_{2n-2}\\
\Phi(L_4)&=&(n+1)!u^2v_{2n-2}.
\end{eqnarray*}

By Lemma~\ref{lemma_im lambda}, $Im(a)$ is spanned by $(1, 0, 0), (0, 2, 0)$ and $(0, 0, 1)$. Under this basis, the coordinates of the $\Phi(L_i)$'s are
\[
\begin{array}{l l}
\Phi(L_1)=(\frac{n!}{2},0,0),
&\Phi(L_2)=(n!,0,0),\\[15pt]
\Phi(L_3)=(0,\frac{n!}{2},\frac{(n+1)!}{2}),
&\Phi(L_4)=(0,0,(n+1)!).
\end{array}
\]
We represent the coordinates of $\Phi(L_i)$'s by the matrix
\[
M_{\Phi}=\frac{n(n-1)}{2}L
\begin{pmatrix}
1	&0	&0\\
2	&0	&0\\
0	&1	&n+1\\
0	&0	&2n+2
\end{pmatrix}
\]
Then $Im(\Phi)$ is spanned by the row vectors of $M_{\Phi}$.

In exact sequence~(\ref{exact seq_necessary Sigma C}) $\KK^0(\Sigma^2C)$ is $\Z\oplus\Z$. Let $\alpha_1$ and $\alpha_2$ be its generators with Chern classes
\[
\begin{array}{l l}
c_{n-1}(\alpha_1)=(n-2)!\Sigma^2v_{2n-4}	&c_n(\alpha_1)=0\\
c_{n-1}(\alpha_2)=0							&c_n(\alpha_2)=(n-1)!\Sigma^2v_{2n-2}.
\end{array}
\]
By Lemma~\ref{lemma_lift of alpha}, $(\partial'_k)_*(\alpha_1)$ and $(\partial'_k)_*(\alpha_2)$ have lifts
\[
\begin{array}{c c c}
\tilde{\alpha}_{1, k}=(n-2)!ku^2v_{2n-4}
&\text{and}
&\tilde{\alpha}_{2, k}=(n-1)!ku^2v_{2n-2}.
\end{array}
\]
We represent their coordinates by a matrix
\[
M_{\partial}=kL
\begin{pmatrix}
1	&0	&0\\
0	&0	&n-1
\end{pmatrix}.
\]
Then the preimage $\tilde{\partial}_k=span\{\tilde{\alpha}_{1, k}, \tilde{\alpha}_{2, k}\}$ of $Im(\partial'_k)_*$ is spanned by the row vectors of $M_{\partial}$. We calculate as in the proof of Lemma~\ref{lemma_im xi 4m+3} to obtain the following lemma.

\begin{lemma}\label{lemma_im xi even}
When $n$ is even, the order of $Im(\partial'_k)_*$ is
\[
|Im(\partial'_k)_*|=\frac{\frac{1}{2}n(n-1)}{(\frac{1}{2}n(n-1), k)}\cdot\frac{n(n+1)}{(n(n+1),k)}.
\]
\end{lemma}

\subsection{Proof of Theorem~\ref{thm_necessary condition}}

Before comparing the orders of $Im(\partial'_k)_*$ and $Im(\partial'_k)_*$, we prove a preliminary lemma.

\begin{lemma}\label{lemma_gcd prime argument}
Let $n$ be an even number and let $p$ be a prime. Denote the $p$-component of $t$ by $\nu_p(t)$. If there are integers $k$ and $l$ such that
\[
\nu_p(\frac{1}{2}n,k)\cdot\nu_p(n,k)=\nu_p(\frac{1}{2}n,l)\cdot\nu_p(n,l),
\]
then $\nu_p(n, k)=\nu_p(n,l)$.
\end{lemma}

\begin{proof}
Suppose $p$ is odd. If $p$ does not divide $n$, then $\nu_p(n,k)=\nu_p(n,l)=1$, so the lemma holds. If $p$ divides $n$, then $\nu_p(\frac{1}{2}n,k)=\nu_p(n,k)$. The hypothesis becomes $\nu_p(n,k)^2=\nu_p(n,l)^2$, implying that $\nu_p(n, k)=\nu_p(n,l)$.

Suppose $p=2$. Let $\nu_2(n)=2^r$, $\nu_2(k)=2^t$ and $\nu_2(l)=2^s$. Then the hypothesis implies
\begin{equation}\label{eqn_nu lemma}
min(r-1,t)+min(r,t)=min(r-1,s)+min(r,s).
\end{equation}
To show $\nu_2(n,k)=\nu_2(n,l)$, we need to show $min(r,t)=min(r,s)$. Consider the following cases: (1)~$t,s\geq r$, (2)~$t,s\leq r-1$, (3)~$t\leq r-1,s\geq r$ and (4)~$s\leq r-1,t\geq r$.

Case (1) obviously gives $min(r,t)=min(r,s)$. In case (2), when $t,s\leq r-1$, equation~(\ref{eqn_nu lemma}) implies $2t=2s$. Therefore $t=s$ and $min(r,t)=min(r,s)$.

It remains to show cases (3) and (4). For case (3) with $t\leq r-1,s\geq r$, equation~(\ref{eqn_nu lemma}) implies
\[
2t=min(r-1, s)+r.
\]
Since $s\geq r$, $min(r-1,s)=r-1$ and the right hand side is $2r-1$ which is odd. However, the left hand side is even, leading to a contradiction. This implies that this case does not satisfy the hypothesis. Case (4) is similar. Therefore $\nu_2(n, k)=\nu_2(n, l)$ and the asserted statement follows.
\end{proof}

\begin{proof}[Proof of Theorem~\ref{thm_necessary condition}]
In exact sequence~(\ref{exact seq_necessary Sigma C}), $[C,\G_k(\C\PP^2)]$ is $Coker(\partial'_k)_*$. By hypothesis,~$\G_k(\C\PP^2)$ is homotopy equivalent to $\G_l(\C\PP^2)$, so $|Im(\partial'_k)_*|=|Im(\partial'_k)_*|$. The $n$ odd and $n$ even cases are proved similarly, but the even case is harder.

When $n$ is even, by Lemma~\ref{lemma_im xi even} the order of $Im(\partial'_k)_*$ is
\[
|Im(\partial'_k)_*|=\frac{\frac{1}{2}n(n-1)}{(\frac{1}{2}n(n-1),k)}\cdot\frac{n(n+1)}{(n(n+1),k)},
\]
so we have
\begin{equation}\label{eqn_necessary pf condition even}
(\frac{1}{2}n(n-1),k)\cdot(n(n+1),k)=(\frac{1}{2}n(n-1),l)\cdot(n(n+1),l).
\end{equation}
We need to show that 
\begin{equation}\label{eqn_necessry pf target even}
\nu_p(n(n^2-1),k)=\nu_p(n(n^2-1),l)
\end{equation}
for all primes $p$. Suppose $p$ does not divide $\frac{1}{2}n(n^2-1)$. Equation~(\ref{eqn_necessry pf target even}) holds since both sides are 1. Suppose $p$ divides $\frac{1}{2}n(n^2-1)$. Since $n-1$, $n$ and $n+1$ are coprime, $p$ divides only one of them. If $p$ divides $n-1$, then $\nu_p(\frac{1}{2}n,k)=\nu_p(n,k)=\nu_p(n+1,k)=1$. Equation~(\ref{eqn_necessary pf condition even}) implies $\nu_p(n-1,k)=\nu_p(n-1,l)$. Since
\[
\nu_p(n(n^2-1),k)=\nu_p(n-1,k)\cdot\nu_p(n,k)\cdot\nu_p(n+1,k),
\]
this implies equation~(\ref{eqn_necessry pf target even}) holds. If $p$ divides $n+1$, then equation~(\ref{eqn_necessry pf target even}) follows from a similar argument. If $p$ divides $n$, then equation~(\ref{eqn_necessary pf condition even}) implies $\nu_p(\frac{1}{2}n,k)\cdot\nu_p(n,k)=\nu_p(\frac{1}{2}n,l)\cdot\nu_p(n,l)$. By Lemma~\ref{lemma_gcd prime argument} $\nu_p(n, k)=\nu_p(n,l)$, so equation~(\ref{eqn_necessry pf target even}) holds.

When $n$ is odd, by Lemma~\ref{lemma_im xi 4m+3} the order of $Im(\partial'_k)_*$ is
\[
|Im(\partial'_k)_*|=\frac{\frac{1}{2}n(n^2-1)}{(\frac{1}{2}n(n^2-1),k)}\cdot\frac{n}{(n,k)},
\]
so we have
\[
(\frac{1}{2}n(n^2-1),k)\cdot(n,k)=(\frac{1}{2}n(n^2-1),l)\cdot(n,l).
\]
We can argue as above to show that for all primes $p$,
\[
\nu_p(\frac{1}{2}n(n^2-1),k)=\nu_p(\frac{1}{2}n(n^2-1),l).
\]
\end{proof}


\begin{thebibliography}{99}
\bibitem{AB83}
M. Atiyah and R. Bott, {\it The Yang-Mills equations over Riemann surfaces}, Philos. Trans. R. Soc. Lond. Ser. A Math. Phys. Eng. Sci., \textbf{308}, (1983), 523 -- 615.


\bibitem{gottlieb72}
D. Gottlieb, {\it Applications of bundle map theory}, Trans. Amer. Math. Soc., \textbf{171}, (1972), 23 -- 50.

\bibitem{HK03}
H. Hamanaka and A. Kono, {\it On $[X, U(n)]$ when $\mathrm{dim}X$ is $2n$}, J. Math. Kyoto Univ., \textbf{43-2}, (2003), 333 -- 348.

\bibitem{HK06}
H. Hamanaka and A. Kono, {\it Unstable K-group and homotopy type of certain gauge groups}, Proc. Roy. Soc. Edinburgh Sect. A, \textbf{136}, (2006), 149 -- 155.



\bibitem{KKT14a}
D. Kishimoto, A. Kono, S. Theriault, {\it Refined gauge group decomposition}, J. Math. Kyoto Univ., \textbf{54}, (2014), 679 -- 691.

\bibitem{KKT14}
D. Kishimoto, A. Kono, M. Tsutaya, {\it On $p$-local homotopy types of gauge groups}, Proc. Roy. Soc. Edinburgh Sect. A, \textbf{144}, (2014), 149 -- 160.

\bibitem{KTT17}
D. Kishimoto, S. Theriault, M. Tsutaya, {\it The homotopy types of $G_2$-gauge groups}, Topology Appl., \textbf{228}, (2017), 92 -- 107.

\bibitem{kono91}
A. Kono, {\it A note on the homotopy type of certain gauge groups}, Proc. Roy. Soc. Edinburgh Sect. A, \textbf{117}, (1991), 295 -- 297.

\bibitem{KT96}
A. Kono and S. Tsukuda, {\it A remark on the homotopy type of certain gauge groups}, J. Math. Kyoto Univ., \textbf{36}, (1996), 115 -- 121.

\bibitem{lang73}
G. Lang, {\it The evaluation map and EHP sequences}, Pacific J. Math., \textbf{44}, (1973), 201 -- 210.

\bibitem{so16}
T. So, {\it Homotopy types of gauge groups over non-simply-connected closed 4-manifolds}, Glasgow Math J., (2018), https://doi.org/10.1017/S0017089518000241.

\bibitem{theriault10a}
S. Theriault, {\it Odd primary homotopy decompositions of gauge groups}, Algebr. Geom. Topol., \textbf{10}, (2010), 535 -- 564.

\bibitem{theriault10b}
S. Theriault, {\it The homotopy types of $Sp(2)$-gauge groups}, J. Math. Kyoto Univ., \textbf{50}, (2010), 591 -- 605.

\bibitem{theriault12}
S. Theriault, {\it Homotopy types of $SU(3)$-gauge groups over simply connected 4-manifolds}, Publ. Res. Inst. Math. Sci., \textbf{48}, (2012), 543 -- 563.

\bibitem{theriault15}
S. Theriault, {\it The homotopy types of $SU(5)$-gauge groups}, Osaka J. Math., \textbf{1}, (2015), 15 -- 31.

\bibitem{theriault17}
S. Theriault, {\it Odd primary homotopy types of $SU(n)$-gauge groups}, Algebr. Geom. Topol., \textbf{17}, (2017), 1131 -- 1150.
\end{thebibliography}
\end{document}